\documentclass[11pt,reqno]{amsart}
\usepackage{enumerate, latexsym, amsmath, amsfonts, amssymb, amsthm, color}
\def\pmod #1{\ ({\rm{mod}}\ #1)}
\def\Z{\Bbb Z}

\def\Q{\Bbb Q}

\def\R{\Bbb R}
\def\C{\Bbb C}

\def\l{\left}
\def\r{\right}
\def\bg{\bigg}
\def\({\bg(}
\def\){\bg)}
\def\t{\text}
\def\f{\frac}

\def\mo{{\rm{mod}\ }}

\def\ls{\leqslant}
\def\gs{\geqslant}

\def\al{\alpha}

\def\ve{\varepsilon}

\def\eq{\equiv}

\def\da{\delta}
\def\la{\lambda}

\def\ind{{\rm Ind}}

\def\aa{{\boldsymbol \alpha}}
\def\bb{{\boldsymbol \beta}}
\def\th{{\boldsymbol \theta}}

\def\e{{\mathbf e}}
\def\u{{\mathbf u}}
\def\v{{\mathbf v}}
\def\U{{\mathbf U}}
\def\V{{\mathbf V}}

\def\Proof{\noindent{\it Proof}}

\theoremstyle{plain}
\newtheorem{theorem}{Theorem}

\newtheorem{lemma}{Lemma}
\newtheorem{corollary}{Corollary}
\newtheorem{conjecture}{Conjecture}
\theoremstyle{definition}

\theoremstyle{remark}
\newtheorem{remark}{Remark}

\makeatletter
\@namedef{subjclassname@2020}{%
  \textup{2020} Mathematics Subject Classification}
\makeatother

\allowdisplaybreaks

 \vspace{4mm}

\begin{document}
%\tableofcontents
\hbox{Preprint, arXiv:2409.08213}
\medskip

\title
[On determinants involving $(\frac{j+k}p)\pm(\frac{j-k}p)$]
{On determinants involving $(\frac{j+k}p)\pm(\frac{j-k}p)$}

\author
[Deyi Chen and Zhi-Wei Sun] {Deyi Chen and Zhi-Wei Sun}

\address {(Deyi Chen)  18 Zheda Road, Xihu District, Hangzhou 310013, Zhejing Province, People's Republic of China}
\email{deyi@zju.edu.cn}

\address{(Zhi-Wei Sun, corresponding author) School of Mathematics, Nanjing
University, Nanjing 210093, People's Republic of China}
\email{zwsun@nju.edu.cn}

\keywords{Determinants, Legendre symbols, linear algebra, Dirichlet characters modulo primes.
\newline \indent The second author is supported by the Natural Science Foundation of China (grant 12371004).}

\begin{abstract} Let $p=2n+1$ be an odd prime.
In this paper, we mainly evaluate determinants involving $(\frac {j+k}p)\pm(\frac{j-k}p)$, where $(\frac{\cdot}p)$ denotes the Legendre symbol. When $p\eq1\pmod4$, we determine the characteristic polynomials of the matrices
$$\left[\left(\frac{j+k}p\right)+\left(\frac{j-k}p\right)\right]_{1\ls j,k\ls n}
\ \text{and}\ \left[\left(\frac{j+k}p\right)-\left(\frac{j-k}p\right)\right]_{1\ls j,k\ls n},$$
and also establish the general identity
\begin{align*}
&\ \left|x+\left(\frac{j+k}p\right)+\left(\frac{j-k}p\right)+\left(\frac jp\right)y+\left(\frac kp\right)z+\left(\frac{jk}p\right)w\right|_{1\ls j,k\ls n}
\\=&\ (-p)^{(p-5)/4}\left(\left(\frac{p-1}2\right)^2wx-\left(\frac{p-1}2y-1\right)\left(\frac{p-1}2z-1\right)\right).
\end{align*}
\end{abstract}
\maketitle

\section{Introduction}
\setcounter{lemma}{0}
\setcounter{theorem}{0}
\setcounter{corollary}{0}
\setcounter{remark}{0}
\setcounter{equation}{0}

Let $p$ be an odd prime, and let $(\frac{.}{p})$ be the Legendre symbol.
In 1956, Lehmer \cite{L} found all the eigenvalues of the matrices
$$\l[x+\l(\f{j+k+m}p\r)\r]_{1\ls j,k\ls p-1}
\ \t{and}\ \l[a+b\l(\f jp\r)+c\l(\f kp\r)+d\l(\f{jk}p\r)\r]_{1\ls j,k\ls p-1},$$
where $m$ is an integer and $a,b,c,d$ are complex numbers.
Lehmer's work was extended by Carlitz \cite{C} in 1959.
In 2004, motivated by coding theory, Chapman \cite{C04} evaluated the determinants
$$\l|x+\l(\f{-1}p\r)\l(\f{j+k}p\r)\r|_{1\ls j,k\ls(p-1)/2}=\l|x+\l(\f{j+k-1}p\r)\r|_{1\ls j,k\ls(p-1)/2}$$
and $$\l|x+\l(\f{-1}p\r)\l(\f{j+k}p\r)\r|_{0\ls j,k\ls(p-1)/2}=\l|x+\l(\f{j+k-1}p\r)\r|_{1\ls j,k\ls(p+1)/2}$$
by using tools including Dirichlet's class number formula.
Chapman's  conjecture on the value of his ``evil" determinant
$|(\f{j-k}p)|_{0\ls j,k\ls (p-1)/2}$
was confirmed by Vsemirnov \cite{V12,V13} in 2012-2013 via quadratic Gauss sums and matrix decomposition; recently Wang, Wu and Ni \cite{WWN}
extended this by evaluating $$\l|x+\l(\f{j-k}p\r)\r|_{0\ls j,k\ls (p-1)/2}.$$

Motivated by the above work and the paper \cite{S19},
Sun \cite{S24} investigated determinants whose entries are linear combinations of Legendre symbols, and posed some interesting conjectures.  For example, Sun \cite[Remark 1.1]{S24}
conjectured that
\begin{equation*}\l|x+\l(\f{j^2+k^2}p\r)+\l(\f{j^2-k^2}p\r)\r|_{1\ls j,k\ls (p-1)/2}
=\l(\f{p-1}2x-1\r)p^{(p-3)/4}
\end{equation*}
for any prime $p\eq3\pmod4$, which was recently confirmed by Li and Wu \cite{LW} via Jacobi sums.

Recall that for a matrix $A=[a_{jk}]_{1\ls j,k\ls n}$ over a field, its characteristic polynomial
is given by
$f_A(x)=|xI_n-A|$, where $I_n$ is the identity matrix of order $n$.

Now we state our first theorem which provides evaluations of the determinants
$$\l|\l(\f{j+k}p\r)+\l(\f{j-k}p\r)\r|_{1\ls j,k\ls(p-1)/2}\
\t{and}\ \l|\l(\f{j+k}p\r)-\l(\f{j-k}p\r)\r|_{1\ls j,k\ls(p-1)/2}$$
for any odd prime $p$.

\begin{theorem} \label{Th1.1} Let $p$ be an odd prime, and set
\begin{equation}\label{A+}A_+=\l[\l(\f{j+k}p\r)+\l(\f{j-k}p\r)\r]_{1\ls j,k\ls(p-1)/2}
\end{equation}  and
\begin{equation}\label{A-}A_-=\l[\l(\f{j+k}p\r)-\l(\f{j-k}p\r)\r]_{1\ls j,k\ls(p-1)/2}.
\end{equation}

{\rm (i)} Suppose that $p\eq1\pmod4$. Then
the characteristic polynomials of the matrices $A_+$ and $A_-$
are
\begin{equation}\label{A+-}f_{A_+}(x)=(x^2-1)(x^2-p)^{(p-5)/4}
\ \t{and}\ f_{A_-}(x)=(x^2-p)^{(p-1)/4},
\end{equation}
respectively. Consequently,
\begin{equation}\label{detA+-}|A_+|=\l(\f 2p\r)p^{(p-5)/4}
\ \ \t{and}\ \ |A_-|=\l(\f 2p\r)p^{(p-1)/4}.
\end{equation}

{\rm (ii)} When $p>3$ and $p\eq3\pmod4$, we have
\begin{equation}\label{Ap34}|A_+|=|A_-|=(-1)^{(h(-p)-1)/2}p^{(p-3)/4},
\end{equation}
where $h(-p)$ denotes the class number of the imaginary quadratic field $\Q(\sqrt{-p})$.
\end{theorem}
\begin{remark} Sun \cite[Conjecture 3.8]{S24}
conjectured that \eqref{detA+-} holds for any prime $p\eq1\pmod4$,  and \eqref{Ap34}
holds for any prime $p>3$ with $p\eq3\pmod 4$.
\end{remark}

Our following theorem provides a further extension of Theorem \ref{Th1.1}.

\begin{theorem}\label{Th1.2}
Let $p>3$ be a prime.

{\rm (i)} Suppose that $p\eq1\pmod4$. Then
\begin{equation}\label{nice}\begin{aligned}
&\ \l|x+\l(\f{j+k}p\r)+\l(\f{j-k}p\r)+\l(\f jp\r)y+\l(\f kp\r)z+\l(\f{jk}p\r)w\r|_{1\ls j,k\ls (p-1)/2}
\\=&\ (-p)^{(p-5)/4}\l(\l(\f{p-1}2\r)^2wx-\l(\f{p-1}2y-1\r)\l(\f{p-1}2z-1\r)\r).
\end{aligned}\end{equation}

{\rm (ii)} When $p\eq3\pmod4$, we have
\begin{equation}\label{xyzw34}\begin{aligned}&\ \l|x+\l(\f{j+k}p\r)+\l(\f{j-k}p\r)+\l(\f jp\r)y+\l(\f kp\r)z+\l(\f{jk}p\r)w\r|_{1\ls j,k\ls (p-1)/2}\\
=&\ (-1)^{(h(-p)-1)/2}p^{(p-3)/4}\l(1-\f{p-1}{2}y-c_p(w+x)+c_p^2(wx-yz)\r)
\\&\ +(-1)^{(h(-p)-1)/2}p^{(p-7)/4}\l(\f{p-1}2+2(d_p-c_p^2)\r)\l(z+\f{p-1}{2}(wx-yz)\r),
\end{aligned}
\end{equation}
where
\begin{equation}
c_p:=\l(2-\l(\f{2}p\r)\r)h(-p)
\ \ \t{and}\ \ d_p:=\sum_{j,k=1}^{(p-1)/2}\(\f{j^2+jk}{p}\).
\end{equation}
\end{theorem}
\begin{remark} Theorem \ref{Th1.2}(i) was conjectured by Sun \cite[Conjecture 3.8(i)]{S24}.
By \cite[Theorem 3, p.\,346]{BS} or \cite[p.\,238]{IR}, for any prime $p>3$ with $p\eq3\pmod4$
we have
\begin{equation}\label{cp}\sum_{k=1}^{(p-1)/2}\l(\f kp\r)=c_p.
\end{equation}

\end{remark}

It is easy to see that Theorem \ref{Th1.2}(ii) has the following consequence not depending on $d_p$.

\begin{corollary}
For any prime $p>3$ with $p\eq3\pmod4$, we have
\begin{equation*}
\begin{aligned}
&\l|x+\l(\f{j+k}p\r)+\l(\f{j-k}p\r)+\l(\f jp\r)y\r|_{1\ls j,k\ls (p-1)/2}\\
=&\ (-1)^{(h(-p)-1)/2}\ p^{(p-3)/4}\l(1-c_px-\f{p-1}2y\r)
\end{aligned}
\end{equation*}
and
\begin{equation*}
\begin{aligned}
&\l|\l(\f{j+k}p\r)+\l(\f{j-k}p\r)+\l(\f jp\r)y+\l(\f{jk}p\r)w\r|_{1\ls j,k\ls (p-1)/2}\\
=&\ (-1)^{(h(-p)-1)/2}\ p^{(p-3)/4}\l(1-c_pw-\f{p-1}2y\r).
\end{aligned}
\end{equation*}
\end{corollary}

Let $p>3$ be a prime with $p\eq3\pmod4$. Sun \cite[Conjecture 3.8(ii)]{S24} conjectured that
\begin{equation}\label{3.8(ii)}\begin{aligned}&\ \l|x+\l(\f{j+k}p\r)+\l(\f{j-k}p\r)+\l(\f jp\r)y+\l(\f{jk}p\r)w\r|_{1\ls j,k\ls (p-1)/2}
\\=&\ (-1)^{\f{h(-p)+1}2}p^{(p-3)/4}\l(\f{p-1}2y-1+c_p(w+x)-\l(\f 2p\r)\f{16q_p}pwx\r)
\end{aligned}
\end{equation}
for some integer $q_p$ only depending on $p$. Putting $z=0$ in \eqref{xyzw34} we find that
\eqref{3.8(ii)} holds for
$$q_p=\l(\f 2p\r)\f{c_p^2-d_p^2+(d_p+(p-1)/2)^2}{16}.$$
This is an integer under our following conjecture.

\begin{conjecture}\label{Conj1.1} Let $p>3$ be a prime. Then
\begin{equation} d_p\eq\begin{cases}4(1-(-1)^{(p-1)/8})\pmod {16}&\t{if}\ p\eq1\pmod8,
\\-2\pmod{16}&\t{if}\ p\eq5\pmod 8,\\ p+1-(-1)^{(h(-p)-1)/2}h(-p)\pmod {16}&\t{if}\ p\eq3\pmod8,
\\2-(-1)^{(h(-p)-1)/2}h(-p)\pmod{16}&\t{if}\ p\eq7\pmod8.
\end{cases}
\end{equation}
\end{conjecture}
\begin{remark} The values of  $d_p$ for odd primes $p<50$ are listed below:
\begin{gather*} d_3=-1,\ d_5=-2,\ d_7=1,\ d_{11}=-5.\ d_{13}=-2,\ d_{17}=0,\ d_{19}=-13,
\\ d_{23}=5,\ d_{29}=-18,\ d_{31}=5,\ d_{37}=-2,\ d_{41}=-8,\ d_{43}=-21,\ d_{47}=13.
\end{gather*}
\end{remark}

Though we are unable to confirm Conjecture \ref{Conj1.1}, we can determine $d_p$ modulo $4$
for each odd prime $p$. Namely, we have the following result.

\begin{theorem}\label{Th1.3} Let $p$ be any odd prime. Then
\begin{equation}\label{dp4}d_p\eq-\f{p-1}2\pmod4.
\end{equation}
\end{theorem}

We are going to prove Theorem \ref{Th1.1} in the next section.
In Section 3 we provide a useful auxiliary theorem. We will prove Theorems \ref{Th1.2} and \ref{Th1.3}
in Section 4. Our tools include linear algebra, the General Matrix-Determinant Lemma and Dirichlet characters.

Based on Sun's work \cite{S24} and our useful auxiliary theorem in Section 3, we are also able to confirm some other conjectures in \cite{S24} including Conjectures 3.1, 3.7(i), 4.1, and Conjecture 3.9 in the case $p\eq1\pmod4$.

\section{Proof of Theorem \ref{Th1.1}}
\setcounter{lemma}{0}
\setcounter{theorem}{0}
\setcounter{corollary}{0}
\setcounter{remark}{0}
\setcounter{equation}{0}

The following lemma is well known, see, e.g., \cite[p.\,58]{BEW}.

\begin{lemma}\label{le3.1}
Let $p$ be an odd prime. For any $b,c\in\Z$ we have
\begin{equation*}
\sum_{x=0}^{p-1}\l(\f{x^2+bx+c}p \r) =
\begin{cases}
p-1 & \t{if}\ p\mid b^2-4c, \\
-1 & \t{otherwise}.
\end{cases}
\end{equation*}
\end{lemma}

As usual, for a matrix $A$ we use $A^T$ to denote its transpose.

\begin{lemma}\label{le5.1}
Let $p=2n+1$ be an odd prime. Define $A_+$ and $A_-$ as in Theorem \ref{Th1.1}.
Then
\begin{equation}\label{A+2}A_+^TA_+=pI_n-\l[2+(1+(-1)^n)\l(\f{jk}p\r)\r]_{1\ls j,k\ls n}
\end{equation}
and
\begin{equation}\label{A-2}A_-^TA_-=pI_n-\l[(1-(-1)^n)\l(\f{jk}p\r)\r]_{1\ls j,k\ls n}.
\end{equation}
\end{lemma}
\begin{proof} Let $\ve\in\{\pm1\}$ and
$$A(\ve)=\l[\l(\f{j+k}p\r)+\ve\l(\f{j-k}p\r)\r]_{1\ls j,k\ls n}.$$
Then $A(\ve)^TA(\ve)=[b_{jk}]_{1\ls j,k\ls n}$, where
\begin{equation*}
b_{jk}=\sum_{i=1}^n\l(\l(\f{i+j}p\r)+\ve\l(\f{i-j}p\r)\r)\l(\l(\f{i+k}p\r)+\ve\l(\f{i-k}p\r)\r).
\end{equation*}
When $j=k\in\{1,\ldots,n\}$, it follows from Lemma \ref{le3.1} that
\begin{align*}
b_{jj}&=\sum_{i=1}^n\(\l(\f{i+j}p\r)+\ve\l(\f{i-j}p\r)\)^2\\
&=\sum_{i=1}^n\l(\f{i+j}p\r)^2+\sum_{i=1}^n\l(\f{i-j}p\r)^2+2\ve\sum_{i=1}^n\l(\f{i^2-j^2}p\r)\\
&=n+(n-1)+\ve\sum_{i=-n}^n\l(\f{i^2-j^2}p\r)-\ve\l(\f{-j^2}p\r)\\
&=n+(n-1)+\ve(-1)-\ve\l(\f{-1}p\r)=p-2-\ve(1+(-1)^n).
\end{align*}
Similarly, when $j\neq k$, Lemma \ref{le3.1} implies that
\begin{align*}
b_{jk}&=\sum_{i=1}^n\(\l(\f{i^2+(j+k)i+jk}p\r)+\l(\f{i^2-(j+k)i+jk}p\r)\)\\
&\quad+\ve\sum_{i=1}^n\(\l(\f{i^2+(j-k)i-jk}p\r)+\l(\f{i^2-(j-k)i-jk}p\r)\)\\
&=\sum_{i=-n}^n\l(\f{i^2+(j+k)i+jk}p\r)-\l(\f{jk}p\r)
\\&\quad+\ve\sum_{i=-n}^n\l(\f{i^2+(j-k)i-jk}p\r)-\ve\l(\f{-jk}p\r)\\
&=-1-\l(\f{jk}p\r)-\ve-\ve(-1)^n\l(\f{jk}p\r)
\\&=-(1+\ve)-(1+(-1)^n\ve)\l(\f{jk}p\r).
\end{align*}
Therefore, both \eqref{A+2} and \eqref{A-2} hold.
\end{proof}
\begin{lemma}\label{Lem2.3} Let $p=2n+1$ be a prime with $p\eq1\pmod4$.
Then both $1$ and $-1$ are eigenvalues of  the matrix $A_+$ given by \eqref{A+}.
Moreover,
\begin{equation}\label{v12}A_+\v_1=\v_1\ \t{and}\ A_+\v_2=-\v_2,
\end{equation}
where
$$\v_1=\left[\(\f{1}p\)-1,\(\f{2}p\)-1,\cdots,\(\f{n}p\)-1\right]^T$$
and
$$\v_2=\left[\(\f{1}p\)+1,\(\f{2}p\)+1,\cdots,\(\f{n}p\)+1\right]^T.$$
Also, $1$ is an eigenvalue of $A_+^2$ with multiplicity at least $2$.
\end{lemma}
\begin{proof} Note that $(\f{-1}p)=(-1)^n=1$. For $\ve\in\{\pm1\}$ and $1\ls j\ls n$, we have
\begin{align*}
&\ \sum_{k=1}^n\l(\l(\f{j+k}p\r)+\l(\f{j-k}p\r)\r)\l(\(\f{k}p\)-\ve\r)\\
=&\ \sum_{k=1}^n\(\l(\f{k^2+jk}{p}\)+\(\f{k^2-jk}{p}\r)\)-\ve\sum_{k=1}^n\l(\l(\f{j+k}{p}\r)+\l(\f{j-k}{p}\r)\r)\\
=&\ \sum_{k=-n}^n\(\f{k^2+jk}{p}\)-\ve\sum_{k=-n}^n\(\f{j+k}{p}\)+\ve\(\f{j}{p}\)\\
=&\ -1-\ve\cdot0+\ve\(\f{j}{p}\)=\ve\l(\l(\f jp\r)-\ve\r).
\end{align*}
So \eqref{v12} holds. As
\begin{equation}\label{p/2}\l|\l\{1\ls k\ls n:\ \l(\f kp\r)=\ve\r\}\r|=\f12\l|\l\{1\ls k\ls p-1:\ \l(\f kp\r)=\ve\r\}\r|=\f{n}2>0
\end{equation}
for each $\ve=\pm1$, neither $\v_1$ nor $\v_2$ is the zero column vector and hence $1$ and $-1$ are eigenvalues of $A_+$. Note also that $\v_1$ and $\v_2$ are linearly independent over the real field $\R$.
Since $A_+^2\v_1=A_+\v_1=A_+\v_1$ and $A_+^2\v_2=-A_+\v_2=\v_2$, we see that $1$ is an eigenvalue of $A_+^2$ with multiplicity at least two. This concludes the proof.
\end{proof}

\begin{lemma}\label{Lem2.4} Let $p=2n+1$ be a prime with $p\eq1\pmod4$. Then $p$ is an eigenvalue of the matrix $A_+$ given by \eqref{A+} with multiplicity at least $n-2$.
\end{lemma}
\begin{proof} The $n$-dimensional column vectors
$$\e_1 = [1, 0, \ldots, 0]^T,\ \e_2 = [0, 1, \ldots, 0]^T, \ldots,\ \e_n = [0, 0, \ldots, 1]^T $$
form a standard orthogonal basis of $\R^n$. As \eqref{p/2} holds for $\ve=\pm1$, we may write
$$\{1,\ldots,n\}=\{s_0,\ldots,s_{n/2-1}\}\cup\{t_0,\ldots,t_{n/2-1}\},$$
where $s_0,\ldots,s_{n/2-1}$ are quadratic residues modulo $p$, and $t_0,\ldots,t_{n/2-1}$
are quadratic nonresidues modulo $p$.

Note that the $n-2$ vectors
$$\aa_i=\e_{s_i}-\e_{s_0}\ \t{and}\
\bb_i=\e_{t_i}-\e_{t_0}\ \ \l(i=1,\ldots,\f n2-1\r)$$
are linearly independent over the field $\R$. We claim that $A_+^2\aa_i=p\aa_i$ and $A_+^2\bb_i=p\bb_i$
for each $i=1,\ldots,n/2-1$.
By \eqref{A+2} we have $A_+^2=[b_{jk}]_{1\ls j,k\ls n}$ with $b_{jk}=p\da_{jk}-2-2(\f{jk}p)$.
Thus, for each $j=1,\ldots,n$, the $j$th component of the column vector $A_+^2\aa_i$ is
\begin{align*}
(A_+^2\aa_i)_j&=[b_{j1},\cdots,b_{jn}]\aa_i\\
&=[b_{j1},\cdots,b_{jn}](\e_{s_i}-\e_{s_0})=b_{js_i}-b_{js_0}\\
&=p(\da_{js_i}-\da_{js_0})-2\l(\l(\f{js_i}p\r)-\l(\f{js_0}p\r)\r)
\\&=p(\da_{js_i}-\da_{js_0})=
\begin{cases}
    -p & \text{if } j=s_0, \\
    p & \text{if } j=s_i,\\
    0 & \text{if } j\neq s_0,s_i,\\
\end{cases}
\end{align*}
which coincides with the $j$th component of $\aa_i$.
So
$A_+^2\aa_i=p\aa_i.$
Similarly,  $A_+^2\bb_i=p\bb_i.$
This proves the claim.

By the last paragraph, $p$ is indeed an eigenvalue of $A_+^2$ with multiplicity at least $n-2$.
This concludes the proof.
\end{proof}

\medskip
\noindent{\tt Proof of the First Part of Theorem \ref{Th1.1}}. Let $n=(p-1)/2$. As $(\f{-1}p)=(-1)^n=1$, both $A_+$ and $A_-$ are symmetric.
Thus $A_+^T=A_+$ and $A_-^T=A_-$.

(i) We first deal with the matrix $A_+$.

Combining Lemmas \ref{Lem2.3} and \ref{Lem2.4}, we see that the characteristic polynomial of $A_+^2$
is
\begin{equation}\label{A^2-value}f_{A_+^2}(x)=(x-1)^2(x-p)^{n-2}.
\end{equation}
In particular,
\begin{equation}\label{|A+|^2}|A_+|^2=(-1)^nf_{A_+^2}(0)=p^{n-2}=p^{(p-5)/2}.
\end{equation}

If $\lambda$ is an eigenvalue of $A_+$ with an associated eigenvector $\v$, then
$$A_+^2\v=A_+(\lambda \v)=\lambda^2\v$$
and hence $\lambda^2$ is an eignevalue of $A_+^2$, thus $\lambda\in\{\pm1,\pm\sqrt p\}$
in view of \eqref{A^2-value}. Let $n_1,n_2,m_1,m_2$ be the multiplicity of $1,-1,\sqrt p,-\sqrt p$
as an eignevalue of $A_+$, respectively. Then  $n_1,n_2\gs1$ by Lemma \ref{Lem2.3}, and
$$|A_+|=1^{n_1}(-1)^{n_2}(\sqrt p)^{m_1}(-\sqrt p)^{m_2}.$$
Combining this with \eqref{|A+|^2}, we find that
\begin{equation}\label{mm}m_1+m_2=n-2=\f{p-5}2.
\end{equation} As $n_1+n_2+m_1+m_2=n$,
we must have $n_1=n_2=1$. Thus the characteristic polynomial of $A_+$ is
$$f_{A_+}(x)=(x^2-1)(x-\sqrt p)^{m_1}(x+\sqrt p)^{m_2}.$$
It follows that the coefficient of $x^{n-1}$ of $f_{A_+}(x)$ is
$$1+(-1)+m_1(-\sqrt p)+m_2\sqrt p.$$
Since this should be an integer, we get $m_1=m_2$. Combining this with \eqref{mm}, we obtain
$m_1=m_2=(p-5)/4$. Therefore $f_{A_+}(x)=(x^2-1)(x^2-p)^{(p-5)/4}$, and hence
$$|A_+|=(-1)^nf_{A_+}(0)=-(-p)^{(p-5)/4}=\l(\f 2p\r) p^{(p-5)/4}.$$

(ii) Now we turn to the matrix $A_-$.
By \eqref{A-2},
 $$f_{A_-^2}(x)=|xI_n-pI_n|=(x-p)^n\ \t{and hence}\ |A_-|^2=p^n.$$
 So all the eigenvalues of $A_-$ belong to the set $\{\pm \sqrt p\}$.
 Let $m_1$ and $m_2$ be the multiplicity of $\sqrt p$ and $-\sqrt{p}$ as an eigenvalue of $A_-$, respectively. Then $m_1+m_2=n=(p-1)/2$, and
 $$f_{A_-}(x)=(x-\sqrt p)^{m_1}(x+\sqrt p)^{m_2}.$$
 Thus the coefficient of $x^{n-1}$ in $f_{A_-}(x)$ is $m_1(-\sqrt p)+m_2\sqrt p$, which should be an integer. Therefore $m_1=m_2=(p-1)/4$, and hence $f_{A_-}(x)=(x^2-p)^{(p-1)/4}$.
 It follows that $|A_-|=(-1)^nf_{A_-}(0)=(\f 2p)p^{(p-1)/4}$.

  In view of the above, we have proved the first part of Theorem \ref{Th1.1}. \qed

  To prove the second part of Theorem \ref{Th1.1}, we need one more lemma.

 \begin{lemma} \label{Lem2.5} Let $p$ be a prime with $p\eq3\pmod4$. For the matrix
 \begin{equation}\label{Ap}A_p=\l[\l(\f{j^2+jk}p\r)+\l(\f{j^2-jk}p\r)\r]_{1\ls j,k\ls (p-1)/2},
 \end{equation}
 we have $|A_p|<0$.
 \end{lemma}
 \Proof. Fix a primitive root $g\in\Z$ modulo $p$, and define a Dirichlet character $\chi$ modulo $p$ by
 setting $\chi(a)=e^{2\pi \mathrm{i}(\ind_g a)/(p-1)}$ for any integer $a\not\eq0\pmod p$,
 where the index $\ind_g a$
 denotes the unique number $r\in\{0,\ldots,p-2\}$ with $g^r\eq a\pmod p$. Then
 $\chi$ is a generator of Dirichlet characters modulo $p$.

Let $n=(p-1)/2$ and $r\in\{1,\ldots,n\}$. For each $j=1,\ldots,n$, we have
\begin{align*}&\ \sum_{k=1}^n\l(\l(\f{j^2+jk}p\r)+\l(\f{j^2-jk}p\r)\r)\chi^r(k^2)
\\=&\ \sum_{k=1}^{p-1}\l(\f{j^2+jk}p\r)\chi^r(k^2)=\sum_{i=1}^{p-1}\l(\f{j^2+j(ji)}p\r)\chi^r((ji)^2)
\\=&\ \(\sum_{i=1}^{p-1}\l(\f{1+i}p\r)\chi^r(i^2)\)\chi(j^2).
\end{align*}
Thus $A_p\v_r=\lambda_r\v_r$, where
$$\lambda_r=\sum_{k=1}^{p-1}\l(\f{k+1}p\r)\chi^r(k^2)
\ \ \t{and}\ \ \v_r=[\chi^r(1^2),\ldots,\chi^r(n^2)]^T.$$

In view of the known evaluation of a determinant of Vandermonde's type, we have
\begin{align*}
\det[\v_1,\v_2,\cdots,\v_n]&=\begin{vmatrix}
\chi^1(1^2) & \chi^2(1^2) & \cdots & \chi^n(1^2) \\
\chi^1(2^2) & \chi^2(2^2) & \cdots & \chi^n(2^2) \\
\vdots & \vdots & \ddots & \vdots \\
\chi^1(n^2) & \chi^2(n^2) & \cdots & \chi^n(n^2)
\end{vmatrix}
 \\
&=\chi(1^2)\chi(2^2)\cdots\chi(n^2)\prod_{1\ls j<k\ls n}\l(\chi(k^2)-\chi(j^2)\r)\\
&\neq0.
\end{align*}
 Thus the vectors $\v_1,\v_2,\cdots,\v_n$ are linearly independent over the field $\C$ of complex numbers. Combining this with the conclusion in the last paragraph, we obtain that
 $$f_{A_p}(x)=\prod_{r=1}^n(x-\la_r)\in\Z[x].$$
 As $\chi^n(k^2)=\chi^{2n}(k)=1$ for each $k=1,\ldots,n$, we have
$$\la_n =\sum_{k=1}^{p-1}\l(\f{k+1}p\r)= \sum_{j=1}^p\l(\f jp\r)-\l(\f1p\r)=-1.$$

Suppose that $\la_r\in\R$ for some $1\ls r\ls n-1$.
As $\chi^r(1^2)=1\in\R$, all the other components of the eigenvector $\v_r$ should be real.
Thus $\chi^{2r}(k)=\chi^r(k^2)\in \R$ for all $k=1,\ldots,n$, and hence $\chi^{2r}(a)\in\R$
for any integer $a\not\eq0\pmod p$. (Note that $\chi^{2r}(-1)\in\{\pm1\}$.)
In particular, $\chi^{2r}(g)\in\R$, i.e., $p-1\mid 4r$. As $n=(p-1)/2$ is odd, we must have $n\mid r$
and hence $r=n$. This leads to a contradiction.

By the above, $\la_n$ is the only real eigenvalue, and the non-real eigenvalues $\la_1,\ldots,\la_{n-1}$ can be divided into $(n-1)/2$ conjugate pairs. Therefore
$|A_p|=\prod_{r=1}^n\la_r$ has the same sign with $\la_n=-1$. This concludes our proof. \qed

\medskip
\noindent{\tt Proof of the Second Part of Theorem \ref{Th1.1}}. Let $n=(p-1)/2$. Then $(\f{-1}p)=(-1)^n=-1$. As $A_+^T=A_-$, we have $|A_+|=|A_-|$.

In light of \eqref{A+2},
$$|A_+|^2=|A_+^TA_+|=|p\da_{jk}-2|_{1\ls j,k\ls n}.$$
Adding columns $2,\ldots,n$ of $P=[p\da_{jk}-2]_{1\ls j,k\ls n}$ to the first column, we obtain the matrix $Q=[q_{jk}]_{1\ls j,k\ls n}$, where
$$q_{jk}=\begin{cases}1&\t{if}\ k=1,\\ p\da_{jk}-2&\t{if}\ k>1.\end{cases}$$
 each row of $Q$ after the first row minus the first row yields a triangular matrix whose
 diagonal elements are $1,p,\ldots,p$. Therefore
 $$|A_+|^2=|p\da_{jk}-2|_{1\ls j,k\ls n}=|Q|=1\times p^{n-1}=p^{(p-3)/2}$$
 and hence $|A_+|=\pm p^{(p-3)/4}$.

For the matrix $A_p$ given by \eqref{Ap}, clearly
$$|A_p|=\prod_{j=1}^n\l(\f jp\r)|A_+|=\l(\f{n!}p\r)|A_+|.$$
As $p>3$ and $p\eq3\pmod4$,  we have
$$n!\eq (-1)^{(h(-p)+1)/2}\pmod p$$
by a result of Mordell \cite{M}. Therefore
$$|A_p|=(-1)^{(h(-p)+1)/2}|A_+|.$$
As $|A_p|<0$ by Lemma \ref{Lem2.5}, we must have $|A_+|=(-1)^{(h(-p)-1)/2}p^{(p-3)/4}$.
This concludes our proof. \qed

\section{An auxiliary theorem}
\setcounter{lemma}{0}
\setcounter{theorem}{0}
\setcounter{corollary}{0}
\setcounter{remark}{0}
\setcounter{equation}{0}

In this section we establish the following useful auxiliary theorem.

\begin{theorem}\label{Th3.1} Let $A=[a_{jk}]_{m\ls j,k\ls n}$ be a matrix over a field $F$ with $|A|=\al\not=0$,
and let $f(j)$ and $g(j)$ be elements of $F$ for all $j=m,\ldots,n$. For the matrix
 \begin{equation}\label{3.1}A(x,y,z,w)=\l[a_{jk}+x+f(j)y+g(k)z+f(j)g(k)w\r]_{m\ls j,k\ls n},
 \end{equation}
 we have
 \begin{equation}\label{3.2}\begin{aligned}
|A(x,y,z,w)|=&\ \alpha(1-x-y-z-w)+(\al_1x+\al_2y+\al_3z+\al_4w)
            \\ &+\(\al_1-\al_2-\al_3+\al_4+\frac{\al_2\al_3-\al_1\al_4}{\alpha}\)(yz-wx),
\end{aligned}\end{equation}
where
\begin{align*}
\al_1=|A(1,0,0,0)|,\, \al_2=|A(0,1,0,0)|,\, \al_3=|A(0,0,1,0)|,\, \al_4=|A(0,0,0,1)|.
\end{align*}
\end{theorem}

To prove Theorem \ref{Th3.1}, we need the following known result (cf. \cite{P}).

\begin{lemma}[General Matrix-Determinant Lemma]\label{MDL}
Let $A$ is an invertible $n\times n$ matrix over a field $F$,  and let $U$ and $V$ be two $n\times m$ matrices over $F$.  Then
\begin{equation}\label{3.3}|A+UV^T|=|I_m+V^TA^{-1}U|\cdot|A|.
\end{equation}
\end{lemma}

\medskip
\noindent {\tt Proof of Theorem \ref{Th3.1}}.
Let's define three column vectors with components in $F$:
\begin{align*}
\u_0=[1,\cdots,1]^{T}\in F^{n-m+1},\
\u_1=[f(m),\cdots,f(n)]^{T},
\ \u_2=[g(m),\cdots,g(n)]^{T}.
\end{align*}
Then
\begin{align*}
|A(x,y,z,w)|&=|A+x\u_0\u_0^T+y\u_1\u_0^T+z\u_0\u_2^T+w\u_1\u_2^T|\\
&=|A+[x\u_0,y\u_1,z\u_0,w\u_1][\u_0,\u_0,\u_2,\u_2]^{T}|\\
&=|A+\U\V^T|,
\end{align*}
where $\U=[x\u_0,y\u_1,z\u_0,w\u_1]$ and $\V=[\u_0,\u_0,\u_2,\u_2].$
Thus, by Lemma \ref{MDL} we have
\begin{equation}\label{3.4}|A(x,y,z,w)|=|I_4+\V^TA^{-1}\U|\cdot|A|.
\end{equation}
It is easy to verify that
\begin{equation}\label{3.5}
|I_4+\V^TA^{-1}\U|=\begin{vmatrix}
ax+1 & by & az & bw \\
ax & by+1 & az & bw \\
cx & dy & cz+1 & dw \\
cx & dy & cz & dw+1
\end{vmatrix},
\end{equation}
where
\begin{equation}\label{3.6}
\begin{aligned}
a:=\u_0^TA^{-1}\u_0&=\frac{|A(1,0,0,0)|-|A|}{|A|}=\frac{\al_1}{\alpha}-1,\\
b:=\u_0^TA^{-1}\u_1&=\frac{|A(0,1,0,0)|-|A|}{|A|}=\frac{\al_2}{\alpha}-1,\\
c:=\u_2^TA^{-1}\u_0&=\frac{|A(0,0,1,0)|-|A|}{|A|}=\frac{\al_3}{\alpha}-1,\\
d:=\u_2^TA^{-1}\u_1&=\frac{|A(0,0,0,1)|-|A|}{|A|}=\frac{\al_4}{\alpha}-1,
\end{aligned}
\end{equation}
by applying Lemma \ref{MDL} with $m=1$.
Combining \eqref{3.4}-\eqref{3.6}, we immediately obtain the desired \eqref{3.2}.
This ends the proof. \qed

For any odd prime $p$, let $\ve_p$ and $h_p$ denote the fundamental unit and class number
	of the real quadratic field $\mathbb{Q}(\sqrt{p})$ respectively, and write
\begin{equation*}
	 \varepsilon_p^{h_p}=a_p+b_p\sqrt{p}\ \ \t{with}\ a_p,b_p\in\mathbb{Q},
\end{equation*}
and
\begin{equation*}
	 \varepsilon_p^{(2-(\frac{2}{p}))h_p}=a_p'+b_p'\sqrt{p}\ \ \t{with}\ a_p',b_p'\in\mathbb{Q}.
\end{equation*}

The following conjecture is due to Sun \cite[Conjecture 3.1]{S24}.

\begin{conjecture}[\cite{S24}]\label{Conj0(p-1)}
Let $p$ be an odd prime.

{\rm (i)} If $p>3$, then
\begin{equation}\begin{aligned}&\ \l|x+\l(\f{j+k}p\r)+\l(\f jp\r)y+\l(\f kp\r)z+\l(\f{jk}p\r)w\r|_{0\ls j,k\ls(p-1)/2}
\\=&\ \begin{cases}(\f 2p)2^{(p-1)/2}(pb_px+a_p(wx-(y+1)(z+1)))&\t{if}\ p\eq1\pmod4,
\\2^{(p-1)/2}((y+1)(z+1)-wx)&\t{if}\ p\eq3\pmod4.\end{cases}
\end{aligned}
\end{equation}

{\rm (ii)} We have
\begin{equation}\begin{aligned}&\ \l|x+\l(\f{j-k}p\r)+\l(\f jp\r)y+\l(\f kp\r)z+\l(\f{jk}p\r)w\r|_{0\ls j,k\ls(p-1)/2}
\\=&\ \begin{cases}a_p'(wx-(y+1)(z+1))+(\f 2p)pb_p'x&\t{if}\ p\eq1\pmod4,
\\wx+(1+y)(1-z)&\t{if}\ p\eq3\pmod4.\end{cases}
\end{aligned}
\end{equation}
\end{conjecture}

As mentioned in \cite[Remark 3.1]{S24}, Sun \cite{S24} proved the above conjecture
 in the case $wx=0$. Combining this with Theorem \ref{Th3.1}, we see that the conjecture holds
 for general $x,y,z,w$.

\section{Proofs of Theorems \ref{Th1.2} and \ref{Th1.3}}
\setcounter{lemma}{0}
\setcounter{theorem}{0}
\setcounter{corollary}{0}
\setcounter{remark}{0}
\setcounter{equation}{0}

We first prove Theorem \ref{Th1.2}.
For convenience, we set
$$n=(p-1)/2,\ \ A=\l[\l(\f{j+k}p\r)+\l(\f{j-k}p\r)\r]_{1\ls j,k\ls n},$$
and
$$\u_0=[1,1,\cdots,1]^{T}\in\R^n\ \ \t{and}\ \
\u_1=\left[\(\f{1}p\),\(\f{2}p\),\cdots,\(\f{n}p\)\right]^T.$$

\medskip
\noindent{\tt Proof of Theorem \ref{Th1.2}(i)}. Note that
$|A|=(\f 2p)p^{(p-5)/4}\not=0$ by \eqref{detA+-}.
Also,
$$\u_0=\frac{\v_2-\v_1}{2}\ \t{and}\ \u_1=\frac{\v_1+\v_2}{2}$$
by Lemma \ref{Lem2.3}, where $\v_1$ and $\v_2$ are as in Lemma \ref{Lem2.3}. Thus
$$A(-\u_0)=A\frac{\v_1-\v_2}{2}=\frac{\v_1+\v_2}{2}=\u_1,$$
$$A(-\u_1)=A\frac{-\v_1-\v_2}{2}=\frac{\v_2-\v_1}{2}=\u_0.$$
It follows that
\begin{equation*}
A^{-1}\u_1=-\u_0 \t{ and }A^{-1}\u_0=-\u_1.
\end{equation*}
Combining this with Lemma \ref{MDL}, we obtain
\begin{align*}
|A(1,0,0,0)|&=|A|(1+\u_0^TA^{-1}\u_0)=|A|(1-\u_0^T\u_1)=|A|,\\
|A(0,1,0,0)|&=|A|(1+\u_0^TA^{-1}\u_1)=|A|(1-\u_0^T\u_0)=|A|(1-n),\\
|A(0,0,1,0)|&=|A|(1+\u_1^TA^{-1}\u_0)=|A|(1-\u_1^T\u_1)=|A|(1-n),\\
|A(0,0,0,1)|&=|A|(1+\u_1^TA^{-1}\u_1)=|A|(1-\u_1^T\u_0)=|A|,
\end{align*}
where $A(x,y,z,w)$ denotes the matrix $$\l[\l(\f{j+k}p\r)+\l(\f{j-k}p\r)+x+\l(\f{j}p\r)y+\l(\f{k}p\r)z+\l(\f{jk}p\r)w\r]_{1\ls j,k\ls n}.$$
Now applying Theorem \ref{Th3.1} we immediately obtain \eqref{nice}.
This ends our proof. \qed

\medskip
\noindent{\tt Proof of Theorem \ref{Th1.2}(ii)}. Assume that $p>3$ and $p\eq3\pmod4$.
For each $j=1,\ldots,n$, clearly
$$\sum_{k=1}^{n}\(\l(\f{j+k}p\r)+\l(\f{j-k}p\r)\)=\sum_{k=-n}^{n}\l(\f{j+k}p\r)-\l(\f{j}p\r)=-\l(\f jp\r)$$
and
$$\sum_{i=1}^n\l(\f ip\r)\l(\l(\f{i+j}p\r)+\l(\f{i-j}p\r)\r)
=\sum_{i=-n}^n\l(\f{i^2+ij}p\r)=-1$$
with the aid of Lemma \ref{le3.1}.
 Thus $A\u_0=-\u_1$ and $\u_1^TA=-\u_0^T.$
 Note also that
 \begin{equation}\label{u10}\u_1^T\u_0=\sum_{j=1}^n\l(\f jp\r)=c_p
 \end{equation}
 in view of \eqref{cp}.

  Define $\th=[\theta_1,\cdots,\theta_n]^T$
with $$\theta_i=\sum_{k=1}^n\l(\l(\f{i+k}{p}\r)-\l(\f{i-k}{p}\r)-2\l(\f{k}{p}\r)\r)$$
for all $i=1,\ldots,n$.
We claim that
\begin{equation}\label{A-th}
A\th=p\u_0.
\end{equation}

Let $j\in\{1,\ldots,n\}$.
By \eqref{A-2}, we have
$$
\sum_{i=1}^n\(\l(\f{i+j}p\r)-\l(\f{i-j}p\r)\)\(\(\f{i+k}{p}\)-\(\f{i-k}{p}\)\)\\
=p\da_{jk}-2\l(\f{jk}p\r)
$$
for each $k=1,\ldots,n$. Thus
\begin{equation*}\begin{aligned}
&\ \sum_{k=1}^n\sum_{i=1}^n\(\l(\f{j+i}p\r)+\l(\f{j-i}p\r)\)\(\(\f{i+k}{p}\)-\(\f{i-k}{p}\)\)\\
=&\ \sum_{k=1}^n\l(p\da_{jk}-2\l(\f{jk}p\r)\r)
= p-2\(\f{j}{p}\)\sum_{k=1}^n\(\f{k}{p}\).
\end{aligned}\end{equation*}
It follows that the $j$-th component of the vector $A\th$ is
\begin{align*}
(A\th)_j&=\sum_{i=1}^n\(\l(\f{j+i}p\r)+\l(\f{j-i}p\r)\)\theta_i\\
&=\sum_{i=1}^n\(\l(\f{j+i}p\r)+\l(\f{j-i}p\r)\)\sum_{k=1}^n\(\(\f{i+k}{p}\)-\(\f{i-k}{p}\)\)\\
&\quad-2\sum_{i=1}^n\(\l(\f{j+i}p\r)+\l(\f{j-i}p\r)\)\sum_{k=1}^n\(\f{k}{p}\)\\
&=\sum_{k=1}^n\sum_{i=1}^n\(\l(\f{j+i}p\r)+\l(\f{j-i}p\r)\)\(\(\f{i+k}{p}\)-\(\f{i-k}{p}\)\)\\
&\quad-2\(\sum_{i=-n}^n\l(\f{j+i}p\r)-\(\f{j}{p}\)\)\sum_{k=1}^n\(\f{k}{p}\)\\
&=p-2\(\f{j}{p}\)\sum_{k=1}^n\(\f{k}{p}\)+2\(\f{j}{p}\)\sum_{k=1}^n\(\f{k}{p}\)=p.
\end{align*}
This proves \eqref{A-th}.

 Recall that $|A|=(-1)^{(h(-p)-1)/2}p^{(p-3)/4}\not=0$ by \eqref{Ap34}.
Note that $A^{-1}\u_0=p^{-1}\th$ by \eqref{A-th}.
Therefore
\begin{align*}
\u_1^TA^{-1}\u_0&=\f{1}{p}\u_1^T\th
=\f{1}{p}\sum_{j=1}^n\(\f{j}{p}\)\theta_j\\
&=\f{1}{p}\sum_{j=1}^n\(\f{j}{p}\)\sum_{k=1}^n\l(2\l(\f{j+k}{p}\r)-\l(\f{j+k}{p}\r)
-\l(\f{j-k}{p}\r)-2\l(\f kp\r)\r)\\
&=\f{1}{p}\sum_{j=1}^n\(\f{j}{p}\)\(2\sum_{k=1}^n\(\f{j+k}{p}\)-\sum_{k=-n}^n\l(\f{j+k}p\r)+\(\f{j}{p}\)-2\sum_{k=1}^n\(\f{k}{p}\)\)\\
&=\f{2}{p}\sum_{j=1}^n\sum_{k=1}^n\(\f{j^2+jk}{p}\)+\f{n}{p}-\f{2}{p}\sum_{j=1}^n\(\f{j}{p}\)\sum_{k=1}^n\(\f{k}{p}\)
\end{align*}
and hence
\begin{equation}\label{dp}\u_1^TA^{-1}\u_0=\f{n+2(d_p-c_p^2)}p
\end{equation}
with the aid of \eqref{cp}.

 Now let $A(x,y,z,w)$ denote the matrix
$$\l[\l(\f{j+k}p\r)+\l(\f{j-k}p\r)+x+\l(\f jp\r)y+\l(\f kp\r)z+\l(\f{jk}p\r)w\r]_{1\ls j,k\ls n}.$$
In view of Lemma \ref{MDL}, by using \eqref{u10} and \eqref{dp}, we get
\begin{align*}
|A(1,0,0,0)|&=|A|(1+\u_0^TA^{-1}\u_0)=|A|(1-\u_1^T\u_0)=|A|(1-c_p),\\
|A(0,1,0,0)|&=|A|(1+\u_0^TA^{-1}\u_1)=|A|(1-\u_0^T\u_0)=|A|\(1-n\),\\
|A(0,0,1,0)|&=|A|(1+\u_1^TA^{-1}\u_0)=|A|\l(1+\f{n+2(d_p-c_p^2)}p\r),\\
|A(0,0,0,1)|&=|A|(1+\u_1^TA^{-1}\u_1)=|A|(1-\u_1^T\u_0)=|A|(1-c_p).
\end{align*}
Combining this with \eqref{Ap34} and  Theorem \ref{Th3.1}, we immediately obtain the desired \eqref{xyzw34}.
This concludes the proof. \qed

To prove Theorem \ref{Th1.3}, we need the following lemma.

\begin{lemma} Let $p=2n+1$ be an odd prime. Then
\begin{equation}\label{j+k}\sum_{j,k=1}^n\l(\f{j+k}p\r)=\begin{cases}2\sum_{k=1}^n k(\f kp)&\t{if}\ p\eq1\pmod4,
\\-\sum_{k=1}^n(\f kp)&\t{if}\ p\eq3\pmod4.
\end{cases}
\end{equation}
\end{lemma}
\Proof. Observe that
\begin{align*}\sum_{j,k=1}^n\l(\f{j+k}p\r)&=\sum_{l=2}^{p-1}\l(\f lp\r)\sum_{1\ls j\ls n\atop 1\ls l-j\ls n}1
\\&=\sum_{1<l\ls n}\l(\f lp\r)\sum_{j=1}^{l-1}1+\sum_{k=1}^n\l(\f{p-k}p\r)\sum_{j=n+1-k}^n1
\\&=\sum_{l=1}^n(l-1)\l(\f lp\r)+\l(\f{-1}p\r)\sum_{k=1}^n k\l(\f kp\r)
\\&=(1+(-1)^n)\sum_{l=1}^n l\l(\f lp\r)-\sum_{l=1}^n\l(\f lp\r).
\end{align*}
If $p\eq1\pmod4$, then $(-1)^n=1$ and
$$2\sum_{l=1}^n\l(\f lp\r)=\sum_{l=1}^n\l(\l(\f lp\r)+\l(\f{p-l}p\r)\r)=\sum_{k=1}^{p-1}\l(\f kp\r)=0.$$
If $p\eq3\pmod4$, then $1+(-1)^n=0$. So we have the desired \eqref{j+k}. \qed

\medskip
\noindent{\tt Proof of Theorem \ref{Th1.3}}. Set $n=(p-1)/2$, and let $N$ denote the number of ordered pairs
$(j,k)$ with $j,k\in\{1,\ldots, n\}$ such that $(\f jp)=1=(\f{j+k}p)$. Then
\begin{align*}4N&=\sum_{j=1}^n\sum_{k=1}^n\l(1+\l(\f jp\r)\r)\l(1+\l(\f{j+k}p\r)\r)
\\&=\sum_{j=1}^n\sum_{k=1}^n\l(1+\l(\f jp\r)+\l(\f{j+k}p\r)+\l(\f{j(j+k)}p\r)\r)
\\&=n^2+n\sum_{j=1}^n\l(\f jp\r)+\sum_{j,k=1}^n\l(\f{j+k}p\r)+d_p.
\end{align*}
Combining this with \eqref{j+k}, we obtain
\begin{equation}\label{d=} d_p=4N-n^2-n\sum_{j=1}^n\l(\f jp\r)+\begin{cases}(-2)\sum_{k=1}^n k(\f kp)&\t{if}\ p\eq1\pmod4,
\\\sum_{k=1}^n(\f kp)&\t{if}\ p\eq3\pmod4.
\end{cases}
\end{equation}
When $p\eq1\pmod4$, both $n$ and $\sum_{j=1}^n(\f jp)$ are even, hence by \eqref{d=} we have
$$d_p\eq-2\sum_{k=1}^n k\l(\f kp\r)\eq-2\sum_{k=1}^n k=-n^2-n\eq -n\pmod 4.$$
If $p\eq3\pmod4$, then $n$ is odd, and hence by \eqref{d=} we have
$$d_p\eq-1+(1-n)\sum_{j=1}^n\l(\f jp\r)\eq-1+(n-1)n\eq -n\pmod4.$$
Therefore, the desired congruence \eqref{dp4} holds. \qed

%\medskip
%\noindent{\bf Statements and Declarations}. There are no competing interests. This original paper contains no data, and it has not been submitted elsewhere.

\setcounter{conjecture}{0}
\end{document}